\documentclass{amsart}

\usepackage{amsmath,amsthm,amsfonts,amssymb,latexsym,mathrsfs,color}
\usepackage{hyperref}

\newtheorem{theorem}{Theorem}[section]
\newtheorem{lemma}[theorem]{Lemma}
\newtheorem{proposition}[theorem]{Proposition}
\newtheorem{corollary}[theorem]{Corollary}

\theoremstyle{definition}

\theoremstyle{remark}

\newcommand{\sech}{\ensuremath{\mathrm{sech\ }}}

\newcommand{\des}{{\rm des\,}}

\newcommand{\msn}{{\mathcal S}_n}
\newcommand{\sgn}{{\rm sgn\,}}
\newcommand{\rz}{{\rm RZ}}

\newcommand{\seps}{\prec}
\newcommand{\sep}{\preceq}
\newcommand{\R}{{\mathbb R}}

\numberwithin{equation}{section}

\begin{document}

\title{Polynomials with only real zeros and the Eulerian polynomials of type $D$}

\author{Shi-Mei Ma}
\address{School of Mathematics and Statistics, Northeastern University at Qinhuangdao, Hebei 066004,China}
\email{shimeima@yahoo.com.cn(S.-M. Ma)}
\thanks{This work is supported by~NSFC (11126217) and the Fundamental Research Funds for the Central Universities (N100323013)}





\begin{abstract}
A remarkable identity involving the Eulerian polynomials of type $D$ was obtained by
Stembridge (Adv. Math. 106 (1994), p.~280, Lemma 9.1). In this paper we explore
an equivalent form of this identity. We prove Brenti's real-rootedness conjecture for the Eulerian polynomials of type $D$.
\end{abstract}

\maketitle



\section{Introduction}
Let $\msn$ denote the symmetric group of all permutations of $[n]$, where $[n]=\{1,2,\ldots,n\}$.
For a permutation $\pi\in\msn$, we define a {\it descent} to be a position $i$ such that $\pi(i)>\pi(i+1)$. Denote by $\des(\pi)$ the number of descents of $\pi$. Let
\begin{equation*}
A_n(x)=\sum_{\pi\in\msn}x^{\des(\pi)+1}=\sum_{k=1}^nA(n,k)x^{k}.
\end{equation*}
The polynomial $A_n(x)$ is called an {\it Eulerian polynomial}, while $A(n,k)$ is called an {\it Eulerian number}.
Denote by $B_n$ the Coxeter group of type $B$. Elements $\pi$ of $B_n$ are signed permutations of $\pm[n]$ such that $\pi(-i)=-\pi(i)$ for all $i$, where $\pm[n]=\{\pm1,\pm2,\ldots,\pm n\}$.
Let
$${B}_n(x)=\sum_{\pi\in B_n}x^{\des_B(\pi)}=\sum_{k=0}^nB(n,k)x^{k},$$
where
$\des_B=|\{i\in[n]:\pi(i-1)>\pi({i})\}|$ with $\pi(0)=0$.
The polynomial $B_n(x)$ is called an {\it Eulerian polynomial of type $B$}, while $B(n,k)$ is called an {\it Eulerian number of type $B$}.
Denote by $D_n$ the Coxeter group of type $D$. The Coxeter group $D_n$ is the subgroup of $B_n$ consisting of signed permutations $\pi=\pi(1)\pi(2)\cdots\pi(n)$ with an even number of negative entries.
Let
$${D}_n(x)=\sum_{\pi\in D_n}x^{\des_D(\pi)}=\sum_{k=0}^nD(n,k)x^{k},$$
where
$\des_D=|\{i\in[n]:\pi(i-1)>\pi({i})\}|$ with $\pi(0)=-\pi(2)$.
The polynomial $D_n(x)$ is called an {\it Eulerian polynomial of type $D$}, while $D(n,k)$ is called an {\it Eulerian number of type $D$} (see~\cite[A066094]{Sloane} for details).
Below are the polynomials ${D}_n(x)$ for $n\leq 3$:
$$D_0(x)=1,D_1(x)=1,D_2(x)=1+2x+x^2,D_3(x)=1+11x+11x^2+x^3.$$

In 1994,
Stembridge~\cite[Lemma 9.1]{Stembridge94} obtained the following remarkable identity:
\begin{equation}\label{Dnx-recu}
D_n(x)=B_n(x)-n2^{n-1}A_{n-1}(x) \quad {\text for}\quad n\geq 2.
\end{equation}
Let $P_n(x)=A_n(x)/x$.
It is well known that
$$\sum_{n=0}^{\infty}P_n(-1)\frac{x^n}{n!}=1+\tanh (x)$$
and
$$\sum_{n=0}^{\infty}B_n(-1)\frac{x^n}{n!}=\sech (2x)$$
(see~\cite{Hirzebruch08} for instance).
For $n\geq 3$, Chow~\cite[Corollary 6.10]{Chow08} obtained that
\begin{equation}\label{chow}
\sgn D_n(-1)=\begin{cases}
0 & \text{if $n$ is odd},\\
(-1)^{\frac{n}{2}} & \text{if $n$ is even}.
\end{cases}
\end{equation}

This paper is organized as follows.
Section~\ref{Section-2} is devoted to an equivalent form of the identity~\eqref{Dnx-recu}.
In Section~\ref{Section-3}, we prove Brenti's~\cite[Conjecture 5.1]{Brenti94} real-rootedness conjecture for the Eulerian polynomials of type $D$.
\section{Derivative polynomials}\label{Section-2}
In 1995, Hoffman~\cite{Hoffman95} introduced the derivative polynomials for tangent and secant:
\begin{equation*}\label{derivapoly-1}
\frac{d^n}{d\theta^n}\tan \theta=P_n(\tan \theta)\quad {\text and}\quad
\frac{d^n}{d\theta^n}\sec\theta=\sec\theta \cdot Q_n(\tan \theta).
\end{equation*}
Various refinements of the polynomials $P_n(u)$ and $Q_n(u)$ have been pursued by several authors
(see~\cite{Cvijovic09,Cvijovic10,Franssens07,Ma12} for instance).
The derivative polynomials for hyperbolic
tangent and secant are defined by
\begin{equation*}\label{derivapoly-2}
\frac{d^n}{d\theta^n}\tanh \theta=\widetilde{P}_n(\tanh \theta)\quad {\text and}\quad
\frac{d^n}{d\theta^n}\sech \theta=\sech\theta \cdot \widetilde{Q}_n(\tanh \theta).
\end{equation*}
It follows from $\tanh \theta=\mathrm i\tan(\theta/i)$ and $\sech \theta=\sec(\theta/i)$ that
\begin{equation*}\label{pnxqnx}
\widetilde{P}_n(x)=\mathrm i^{n-1}P_n(\mathrm i x) \quad
{\text and}\quad \widetilde{Q}_n(x)=\mathrm i^{n}Q_n(\mathrm i x).
\end{equation*}
From the chain rule it follows that the polynomials
$\widetilde{P}_n(x)$ satisfy
\begin{equation}\label{pnx-recu}
\widetilde{P}_{n+1}(x)=(1-x^2)\widetilde{P}'_n(x)
\end{equation}
with initial values $\widetilde{P}_0(x)=x$.
Similarly, $\widetilde{Q}_0(x)=1$ and
\begin{equation}\label{qnx-recu}
\widetilde{Q}_{n+1}(x)=(1-x^2)\widetilde{Q}'_n(x)-x\widetilde{Q}_n(x).
\end{equation}

Let
$$\tan^k (x)=\sum_{n\geq k}T(n,k)\frac{x^n}{n!}$$
and
$$\sec (x) \tan^k (x)=\sum_{n\geq k}S(n,k)\frac{x^n}{n!}.$$
The numbers $T(n,k)$ and $S(n,k)$ are respectively called the {\it tangent numbers of order $k$} (see~\cite[p.~428]{Carlitz72}) and the {\it secant numbers of order $k$} ((see~\cite[p.~305]{Carlitz75})).
The numbers $T(n,1)$ are sometimes called the {\it tangent numbers} and
$S(n,0)$ are called the {\it Euler numbers}.
Note that the tangent is an odd function and the secant is an even function. Then
$$T(2n,1)=S(2n+1,0)=0,\quad T(2n+1,1)\neq 0 \quad {\text and}\quad S(2n,0)\neq 0 .$$
Recently, Cvijovi\'c~\cite[Theorem 2]{Cvijovic09} showed that
$$\widetilde{P}_n(x)=(-1)^{\frac{n-1}{2}}T(n,1)+\sum_{k=1}^{n+1}\frac{(-1)^{\frac{n+k-1}{2}}}{k}T(n+1,k)x^k$$
and
$$\widetilde{Q}_n(x)=\sum_{k=0}^n(-1)^{\frac{n+k}{2}}S(n,k)x^k.$$
In particular, we have
\begin{equation}\label{pn0qn0}
\widetilde{P}_{2n-1}(0)=(-1)^{n-1}T(2n-1,1) \quad
{\text and}\quad \widetilde{Q}_{2n}(0)=(-1)^nS(2n,0).
\end{equation}
The first few of the polynomials $\widetilde{P}_{n}(x)$ and $\widetilde{Q}_{n}(x)$
are respectively given as follows:
$$ \widetilde{P}_{1}(x) =-x^2+1,  \widetilde{P}_{2}(x)=2x^3-2x, \widetilde{P}_{3}(x)=-6x^4+8x^2-2, \widetilde{P}_{4}(x)=24x^5-40x^3+16x;$$
$$ \widetilde{Q}_{1}(x) =-x,  \widetilde{Q}_{2}(x)=2x^2-1,  \widetilde{Q}_{3}(x)=-6x^3+5x, \widetilde{Q}_{4}(x)=24x^4-28x^2+5.$$

For $n\geq 2$, we define
\begin{equation*}\label{anxbnx-def}
a_n(x)=(x+1)^{n+1}A_{n}\left(\frac{x-1}{x+1}\right),\quad
b_n(x)=(x+1)^nB_n\left(\frac{x-1}{x+1}\right)
\end{equation*}
and
\begin{equation}\label{def-dnx}
d_n(x)=\left(\frac{x+1}{2}\right)^nD_n\left(\frac{x-1}{x+1}\right).
\end{equation}
Then
\begin{equation}\label{anx-bnx-recu}
2^nd_n(x)=b_n(x)-n2^{n-1}a_{n-1}(x)\quad {\text for}\quad n\geq 2.
\end{equation}

From~\cite[Theorem~5,~Theorem~6 ]{Franssens07}, we obtain
\begin{equation}\label{anx}
a_n(x)=(-1)^n\widetilde{P}_{n}(x)\quad {\text and}\quad b_n(x)=(-1)^n2^n\widetilde{Q}_{n}(x).
\end{equation}
Therefore, the polynomials $a_n(x)$ satisfy
the recurrence relation
\begin{equation}\label{anx-recurrence}
a_{n+1}(x)=(x^2-1)a'_n(x)
\end{equation}
with initial values $a_0(x)=x$. The polynomials $b_n(x)$ satisfy
the recurrence relation
\begin{equation}\label{bnx-recurrence}
b_{n+1}(x)=2(x^2-1)b'_n(x)+2xb_n(x)
\end{equation}
with initial values $b_0(x)=1$.
From~\eqref{Dnx-recu}, we get the following result.
\begin{proposition}\label{recu-dnx}
For $n\geq 2$, we have
\begin{equation}\label{dnxpnxqnx}
2d_n(x)={(-1)^n}(n\widetilde{P}_{n-1}(x)+2\widetilde{Q}_{n}(x)).
\end{equation}
\end{proposition}

The first few terms of $d_n(x)$ can be computed directly as follows:
\begin{align*}
  d_2(x)& =x^2,\\
  d_3(x)& =3x^3-2x,\\
  d_4(x)& =12x^4-12x^2+1,\\
  d_5(x)& =60x^5-80x^3+21x,\\
  d_6(x)& =360x^6-600x^4+254x^2-13.
\end{align*}
It follows from~\eqref{pnx-recu} and~\eqref{qnx-recu} that
$d_n(-1)=(-1)^n$ for $n\geq 2$.

\begin{corollary}
For $n\geq 1$, we have $D_{2n-1}(-1)=0$ and
\begin{equation*}\label{D2n}
D_{2n}(-1)=(-4)^n(S(2n,0)-nT(2n-1,1)),
\end{equation*}
where $T(n,1)$ are the tangent numbers and $S(n,0)$ are the Euler numbers.
\end{corollary}
\begin{proof}
Note that $D_{2n-1}(-1)=2^{2n-1}d_{2n-1}(0)$. It is easy to verify that
$\widetilde{P}_{2n-2}(0)=\widetilde{Q}_{2n-1}(0)=0$,
$\widetilde{P}_{2n-1}(0)=(-1)^{n-1}T(2n-1,1)$ and $\widetilde{Q}_{2n}(0)=(-1)^nS(2n,0)$.
Then $D_{2n-1}(-1)=0$.
By~\eqref{def-dnx}, we obtain $D_{2n}(-1)=4^nd_{2n}(0)$.
From~\eqref{dnxpnxqnx}, we obtain $d_{2n}(0)=n\widetilde{P}_{2n-1}(0)+\widetilde{Q}_{2n}(0)$.
Then by~\eqref{pn0qn0}, we get the desired result.
\end{proof}
\section{Main results}\label{Section-3}
Polynomials with only real zeros arise often in combinatorics, algebra and
geometry. We refer the reader to~\cite{Branden06,Chow03,Chow08,Dilks09,Liu07,Visontai12} for various results involving zeros of the polynomials $A_n(x),B_n(x)$ and $D_n(x)$.
This Section is devoted to prove Brenti's~\cite[Conjecture 5.1]{Brenti94} real-rootedness conjecture for the Eulerian polynomials of type $D$.

Let $\rz$ denote the set of real polynomials with only real zeros.
Denote by $\rz(I)$ the set of such polynomials all
whose zeros are in the interval $I$. Suppose that $f,F\in\rz$. Let $\{s_i\}$ and $\{r_j\}$ be all zeros of $F$ and $f$ in nonincreasing
order respectively. Following~\cite{Chudnovsky07}, we say that $F$ {\it interleaves} $f$, denoted by
$f\sep F$, if $\deg f\le\deg F\le\deg f+1$ and
\begin{equation}\label{sep}
s_1\geq r_1\geq s_2\geq r_2\geq s_3\geq r_3\geq\cdots.
\end{equation}
If no equality sign occurs in~\eqref{sep}, then we say that $F$ {\it strictly interleaves} $f$.
Let $f\seps F$ denote $F$ strictly interleaves $f$.

The key ingredient of our proof is the following result due to Hetyei~\cite{Hetyei08}.
\begin{lemma}[{\cite[Proposition 6.5,~Theorem 8.6]{Hetyei08}}]\label{Hetyei}
For $n\geq 1$, we have $\widetilde{P}_{n}(x)\in\rz[-1,1]$, $\widetilde{Q}_{n}(x)\in\rz(-1,1)$ and $\widetilde{Q}_{n}(x)\seps\widetilde{P}_{n}(x)$.
Moreover, $\widetilde{P}_{n-1}(x)\sep\widetilde{P}_{n}(x)$ and
$\widetilde{Q}_{n-1}(x)\sep\widetilde{Q}_{n}(x)$ for $n\geq 2$.
\end{lemma}

By Lemma~\ref{Hetyei}, we obtain
$a_{n-1}(x)\sep a_{n}(x), b_{n-1}(x)\sep b_{n}(x)$ and $b_{n}(x)\seps a_{n}(x)$.
Let $\sgn$ denote the sign function defined on $\R$ by
\begin{equation*}
\sgn x=\begin{cases}
1 & \text{if $x>0$},\\
0 & \text{if $x=0$},\\
-1 & \text{if $x<0$}.
\end{cases}
\end{equation*}
We now present the main result of this paper.
\begin{theorem}\label{mthmzero}
For $n\geq 2$, we have $D_n(x)\in\rz(-\infty,0)$.
\end{theorem}
\begin{proof}
Clearly, $D_n(x)\in\rz(-\infty,0)$ if and only if $d_n(x)\in\rz(-1,1)$.
Since $d_2(x)=x^2$ and $d_3(x)=3x^3-2x$, it suffices to consider the case $n\geq 4$.

Note that the polynomials $a_{n}(x)$ and $b_n(x)$ have the following expressions:
\begin{equation*}\label{any}
a_n(x)=\sum_{k=0}^{\lfloor(n+1)/2\rfloor}(-1)^kp(n,n-2k+1)x^{n-2k+1},
\end{equation*}
\begin{equation*}\label{any}
 b_n(x)=\sum_{k=0}^{\lfloor{n/2}\rfloor}(-1)^kq(n,n-2k)x^{n-2k}.
\end{equation*}
Using Lemma~\ref{Hetyei}, we write
$$a_{2n-1}(x)=(2n-1)!\prod_{i=1}^n(x-s_i)(x+s_i),$$
$$a_{2n}(x)=(2n)!x\prod_{i=1}^n(x-a_i)(x+a_i),$$
$$b_{2n}(x)=(2n)!4^{n}\prod_{j=1}^n(x-r_j)(x+r_j),$$
and
$$b_{2n+1}(x)=(2n+1)!2^{2n+1}x\prod_{j=1}^{n}(x-b_j)(x+b_j),$$
where
\begin{equation}\label{zeros-1}
1=s_1> r_1>s_2> r_2> \cdots >r_{n-1}>s_n>r_n>0
\end{equation}
and
\begin{equation}\label{zeros-2}
1=a_1> b_1>a_2> b_2> \cdots b_{n-1}>a_n>b_n>0.
\end{equation}
Using~\eqref{anx-recurrence} and~\eqref{bnx-recurrence},
the inequalities~\eqref{zeros-1} and~\eqref{zeros-2} for zeros can be easily proved by induction on $n$.
We omit the proof of this for brevity.

By~\eqref{anx-bnx-recu}, we get $$d_{2n}(x)=\frac{b_{2n}(x)}{4^n}-na_{2n-1}(x).$$
Let $F(x)=\prod_{i=1}^n(x-s_i)$ and $f(x)=\prod_{j=1}^n(x-r_j)$.
Then $$d_{2n}(x)=(2n-1)!(-1)^nn\{2f(x)f(-x)-F(x)F(-x)\}.$$
Note that $\sgn d_{2n}(s_{j+1})=(-1)^{j}$ and $\sgn d_{2n}(r_j)=(-1)^{j+1}$, where $1\leq j\leq n-1$. .
Therefore, $d_{2n}(x)$ has precisely one zero in each of $2n-2$ intervals $(s_{j+1},r_j)$ and $(-r_j,-s_{j+1})$
Note that $\sgn d_{2n}(r_n)=(-1)^{n-1}$ and $\sgn d_{2n}(-r_n)=(-1)^{n+1}$.
It follows from~\eqref{chow} that $\sgn d_{2n}(0)=(-1)^n$.
Therefore, $d_{2n}(x)$ has precisely one zero in each of the intervals $(-r_n,0)$ and $(0,r_n)$.
Thus $d_{2n}(x)\in\rz(-1,1)$.

Along the same lines, by~\eqref{anx-bnx-recu}, we get
$$d_{2n+1}(x)=\frac{b_{2n+1}(x)}{2^{2n+1}}-\frac{1}{2}(2n+1)a_{2n}(x).$$
Let $G(x)=\prod_{i=1}^n(x-a_i)$ and $g(x)=\prod_{j=1}^n(x-b_j)$. Then
$$d_{2n+1}(x)=(2n+1)!(-1)^nx\{g(x)g(-x)-\frac{1}{2}G(x)G(-x)\}.$$
Note that $\sgn d_{2n+1}(a_{j+1})=(-1)^{j}$ and $\sgn d_{2n+1}(b_j)=(-1)^{j+1}$, where $1\leq j\leq n-1$.
Therefore, $d_{2n+1}(x)$ has precisely one zero in each of $2n-2$ intervals $(a_{j+1},b_j)$ and $(-b_j,-a_{j+1})$. Note that $\sgn d_{2n+1}(b_n)=(-1)^{n+1}$ and $\sgn d_{2n+1}(-b_n)=(-1)^{n}$.
It follows from~\eqref{dnxpnxqnx} that
$$\sgn \lim_{x \to 0}\frac{d_{2n+1}(x)}{x}=(-1)^{n}.$$
Hence
$$\sgn \lim_{x \to 0^-}d_{2n+1}(x)=(-1)^{n+1}\quad {\text and}\quad \sgn \lim_{x \to 0^+}d_{2n+1}(x)=(-1)^{n}.$$
Therefore, $d_{2n+1}(x)$ has precisely one zero in each of the intervals $(-b_n,0)$ and $(0,b_n)$. Moreover, $d_{2n+1}(x)$ has a simple zero $x=0$.
Thus $d_{2n+1}(x)\in\rz(-1,1)$.

In conclusion, we define
$$d_{2n}(x)=\frac{(2n)!}{2}\prod_{i=1}^n(x-c_i)(x+c_i)$$
and
$$d_{2n+1}(x)=\frac{(2n+1)!}{2}x\prod_{i=1}^n(x-d_i)(x+d_i),$$
where $c_1>c_2>\cdots>c_{n-1}>c_n$ and $d_1>d_2>\cdots>d_{n-1}>d_n$.
Then
\begin{equation}\label{zeros-inequ1}
r_1>c_1>s_2>r_2>c_2>s_3>\cdots >r_{n-1}>c_{n-1}>s_n>r_n>c_n>0
\end{equation}
and
\begin{equation}\label{zeros-inequ2}
b_1>d_1>a_2>b_2>d_2>a_3>\cdots >b_{n-1}>d_{n-1}>a_n>b_n>d_n>0.
\end{equation}
This completes the proof.
\end{proof}

We say that the polynomials $f_1(x),\ldots,f_k(x)$ are {\it compatible} if for all nonnegative real numbers $c_1,c_2,\ldots,c_k$, we have $\sum_{i=1}^kc_if_i(x)\in\rz$. Let $f(x),g(x)\in\rz$. A {\it common interleaver} for
$f(x)$ and $g(x)$ is a polynomial that interleaves $f(x)$ and $g(x)$ simultaneously.
Denote by $n_f(x)$ the number of real zeros of a polynomial $f(x)$ that lie
in the interval $[x,\infty)$ (counted with their multiplicities).
Chudnovsky and Seymour~\cite{Chudnovsky07} established the following two lemmas.
\begin{lemma}[{\cite[3.5]{Chudnovsky07}}]\label{Chudnovsky}
Let $f(x),g(x)\in\rz$. Then $f(x)$ and $g(x)$ have a common interleaver
if and only if $|n_{f}(x)-n_{g}(x)|\leq 1$ for all $x\in \R$.
\end{lemma}
\begin{lemma}[{\cite[3.6]{Chudnovsky07}}]\label{Chudnovsky-2}
Let $f_1(x),f_2(x),\ldots,f_k(x)$ be polynomials with positive leading coefficients and all zeros real. Then
following three statements are equivalent:
\begin{enumerate}
  \item [\rm (a)] $f_1(x),f_2(x),\ldots,f_k(x)$ are pairwise compatible,
  \item [\rm (b)] for all $s,t$ such that $1\leq s<t\leq k$, the polynomials $f_s,f_t$ have a common interleaver,
  \item [\rm (c)] $f_1(x),f_2(x),\ldots,f_k(x)$ are compatible.
\end{enumerate}
\end{lemma}

By~\eqref{zeros-inequ1} and~\eqref{zeros-inequ2}, we obtain
$$|n_{a_{n-1}}(x)-n_{b_{n}}(x)|\leq 1,\quad |n_{a_{n-1}}(x)-n_{d_{n}}(x)|\leq 1$$
and $$|n_{d_{n}}(x)-n_{b_{n}}(x)|\leq 1$$
for all $x\in\R$.
Combining Lemma~\ref{Chudnovsky} and Lemma~\ref{Chudnovsky-2}, we get the following result.
\begin{theorem}\label{compatible}
For $n\geq 2$, the polynomials $a_{n-1}(x),b_{n}(x)$ and $d_n(x)$ are compatible.
Equivalently, the polynomials $A_{n-1}(x),B_{n}(x)$ and $D_{n}(x)$ are compatible.
\end{theorem}

\bibliographystyle{amsplain}

\begin{thebibliography}{19}
\bibitem{Branden06}
P. Br\"and\'en, \textit{On linear transformations preserving the P\'olya frequency property,} Trans. Amer. Math. Soc.
358 (2006), 3697--3716.

\bibitem{Brenti94}
F. Brenti, \textit{q-Eulerian polynomials arising from Coxeter groups}, European J. Combin. 15 (1994), 417--441.

\bibitem{Carlitz72}
L. Carlitz and R. Scoville, \textit{Tangent numbers and operators}, Duke Math. J. 39 (1972), 413--429.

\bibitem{Carlitz75}
L. Carlitz, \textit{Permutations, sequences and special functions}, SIAM Review, 17 (1975), 298--322.


\bibitem{Chow03}
C.-O. Chow, \textit{On the Eulerian polynomials of type D}, European J. Combin. 24 (2003), 391--408.

\bibitem{Chow08}
C.-O. Chow, \textit{On certain combinatorial expansions of the Eulerian polynomials}, Adv. in Appl. Math. 41 (2008), 133--157.

\bibitem{Chudnovsky07}
M. Chudnovsky, P. Seymour, \textit{The roots of the independence polynomial of a clawfree graph},
J. Combin. Theory Ser. B 97 (2007), 350--357.

\bibitem{Cvijovic09}
D. Cvijovi\'c, \textit{Derivative polynomials and closed-form higher derivative formulae},
Appl. Math. Comput. 215 (2009), 3002--3006.

\bibitem{Cvijovic10}
D. Cvijovi\'c, \textit{The Lerch zeta and related functions of non-positive integer order},
Proc. Amer. Math. Soc. 138 (2010), 827--836.

\bibitem{Dilks09}
K. Dilks, T.K. Petersen, J.R. Stembridge, \textit{Affine descents and the Steinberg torus}, Adv Appl. Math.
42 (2009), 423--444.

\bibitem{Franssens07}
G.R. Franssens, \textit{Functions with derivatives given by polynomials in the function itself or a related function}, Anal. Math. 33 (2007), 17--36.

\bibitem{Hetyei08}
G. Hetyei, \textit{Tchebyshev triangulations of stable simplicial complexes},
J. Combin. Theory Ser. A 115 (2008), 569--592.

\bibitem{Hirzebruch08}
F. Hirzebruch, \textit{Eulerian polynomials}, M\"unster J. of Math. 1 (2008), 9--14.

\bibitem{Hoffman95}
M.E. Hoffman, \textit{Derivative polynomials for tangent and secant}, Amer. Math.
Monthly 102 (1995), 23--30.

\bibitem{Liu07}
L. Liu and Y. Wang, \textit{A unified approach to polynomial sequences with only real
zeros}, Adv. in Appl. Math. 38 (2007), 542--560.


\bibitem{Ma12}
S.-M. Ma, \textit{An explicit formula for the number of permutations with a given number of alternating runs},
J. Combin. Theory Ser. A 119 (2012), 1660--1664.

\bibitem{Sloane}
N.J.A. Sloane, \textit{The On-Line Encyclopedia of Integer Sequences},
http://oeis.org.

\bibitem{Stembridge94}
J.R. Stembridge, \textit{Some permutation representations of Weyl groups associated with the cohomology of toric varieties},
Adv. Math. 106 (1994), 244--301.

\bibitem{Visontai12}
M. Visontai, N. Williams, \textit{Stable multivariate $W$-Eulerian polynomials}, arXiv:1203.0791v1.

\end{thebibliography}

\end{document}